\documentclass[12pt]{article}
\usepackage{amsmath, amsthm, amssymb, amsfonts, enumerate}
\usepackage[figuresright]{rotating}
\usepackage{natbib}

\begin{document}
\newcommand{\bea}{\begin{eqnarray}}
\newcommand{\ena}{\end{eqnarray}}
\newcommand{\beas}{\begin{eqnarray*}}
\newcommand{\enas}{\end{eqnarray*}}
\newcommand{\beq}{\begin{equation}}
\newcommand{\enq}{\end{equation}}
\def\qed{\hfill \mbox{\rule{0.5em}{0.5em}}}
\newcommand{\bbox}{\hfill $\Box$}
\newcommand{\From}{From}
\newcommand{\ignore}[1]{}
\newcommand{\ws}{{\widetilde \sigma}}
\newcommand{\btheta}{\mbox{\boldmath {$\theta$}}}
\newcommand{\transpose}{{\mbox{\scriptsize\sf T}}}
\newcommand{\esssup}{\mathop{\mathrm{ess\,sup}}}
\newcommand{\B}{B}
\newcommand{\D}{D}
\newcommand{\E}{E}
\newcommand{\V}{V}
\newcommand{\Km}{K_{\eta,m}}
\newcommand{\Eta}{\eta}
\newcommand{\de}{\varrho}
\newcommand{\tP}{\overline{P}}
\newcommand{\tD}{\overline{D}}
\newcommand{\tQ}{\overline{Q}}
\newcommand{\tL}{\overline{L}}
\newcommand{\tN}{\overline{N}}
\newcommand{\fmx}{\sup}
\newcommand{\abs}[1]{\left\vert#1\right\vert}
\numberwithin{equation}{section}
\theoremstyle{plain}
\newtheorem{thm}{Theorem}[section]
\newtheorem{lem}{Lemma}[section]
\newtheorem{cor}{Corollary}[section]
\newtheorem{asser}{Assertion}[section]


\newcommand{\To}{\rightarrow}
\newcommand{\wtilde}[1]{\widetilde{#1}}
\newcommand{\qmq}[1]{\quad\mbox{#1}\quad}
\newcommand{\changed}{\texttt{[CHANGED:]}}
\newcommand{\rtbct}{]}
\newcommand{\conjB}{[B]}
\def\beq{\begin{equation}}
\def\eeq{\end{equation}}

\newcommand{\figformat}{pdf}

\title{{\bf\Large The Fighter Problem: Optimal Allocation of a\\ Discrete Commodity}}

\author{Jay Bartroff$^1$ and Ester Samuel-Cahn$^2$\\\\
\small{$^1$ Department of Mathematics, University of Southern California, Los Angeles, California, USA}\\
\small{$^2$ Department of Statistics and Center for the Study of
Rationality,}\\ \small{The Hebrew University of Jerusalem,
Israel}\\}  \maketitle

\date{}

\begin{abstract}
The Fighter problem with discrete ammunition is studied. An aircraft (fighter) equipped with $n$ anti-aircraft missiles is intercepted by enemy airplanes, the appearance of which follows a homogeneous Poisson process with known intensity.  If $j$ of the $n$ missiles are spent at an encounter they destroy an enemy plane with probability $a(j)$, where $a(0) = 0 $ and $\{ a(j)\}$ is a known, strictly increasing concave sequence, e.g., $a(j) = 1-q^j, \; \, 0 < q < 1$.  If the enemy is not destroyed, the enemy shoots the fighter down with known probability $1-u$, where  $0 \le u \le 1$.  The goal of the fighter is to shoot down as many enemy airplanes  as possible during a given time period~$[0, T]$.  Let $K (n, t)$ be the smallest optimal number of missiles to be used at a present encounter, when the fighter has flying time~$t$ remaining  and $n$ missiles remaining.  Three seemingly obvious properties of $K(n, t)$ have been conjectured: [A]  The closer to the destination, the more of the $n$ missiles one should use,
[B] the more missiles one has, the more one should use, and [C] the more missiles one has, the more  one should save for possible future encounters.  We show that [C] holds for all $0 \le u \le 1$, that [A] and [B] hold for the ``Invincible Fighter'' ($u=1$), and that [A] holds but [B] fails for the ``Frail Fighter'' ($u=0$), the latter through a surprising counterexample.
\end{abstract}

\section{Introduction}
We first describe two different models, the \textit{Bomber} and \textit{Fighter} models, with the following common probabilistic structure but with differing goals, stated below.  In both models, there are two agents acting: An aircraft which is equipped with $n$ anti-aircraft missiles, and an enemy which sends its intercepting airplanes according to a Poisson process which, without loss of generality, (by adjusting the time scale) we take to have intensity one.  When meeting an enemy, one must decide how many of the $n$ available missiles one should ``spend" on it.  The probability of destroying the enemy when attacking it with $j$ missiles is $a(j)$, where $a(0) = 0 $ and $\{ a(j)\}$ is a strictly increasing and strictly concave sequence, for example, $a(j) = 1 - q^j$ for some $0 < q < 1$.  This  $\{ a(j)\}$ corresponds to a situation where the missiles are i.i.d. Bernoulli with success probability $(1-q)$. If the aircraft fails to destroy the enemy, the enemy counterattacks it, and succeeds with a fixed, known probability
$1-u$, where $0 \le u \le 1$. Let $T > 0 $ and $[0, T]$ be a fixed time period.  The Bomber and Fighter differ only in their goals: They are, respectively,

\begin{description}
\item[Bomber:] To maximize the probability of reaching the target $T$ safely.
\item[Fighter:] To maximize the number of enemy planes shot down during the time period $[0, T]$.
\end{description}
Let $K(n, t)$ be an optimal number of missiles one should spend when meeting an enemy at a time $t$ from the destination, and $n$ missiles are at hand, referred to as \textit{state} $(n, t)$.  Clearly $K(n, t)$ differs for the two models, but it may seem intuitively obvious, or at least reasonable, that for both cases $K(n, t)$ possesses the following three properties:

\begin{description}
\item[[A\rtbct] $K(n, t)$ is non-increasing in $t$ for fixed $n$.
\item[[B\rtbct] $K(n,t)$ is non-decreasing in $n$ for fixed $t$. 
\item[[C\rtbct] $n-K(n,t)$ is non-decreasing in $n$ for fixed $t$. 
\end{description}

The Bomber problem with discrete ammunition was first considered by \citet{Klinger68}.  They assume
that [B] holds, and prove [A] under that assumption. \citet{Samuel70} proves [A] and [C] without any assumption about [B]. The question whether or not [B] holds for the Bomber with discrete ammunition is still not settled.  The problem is also considered in \citet{Simons90}.  For continuous ammunition, both of the problems have recently been considered by \citet{Bartroff10c,Bartroff10} and \citet{Bartroff11}.  In \citet{Bartroff10c}, [A] and [C] are shown to hold for the Bomber problem with continuous ammunition, for which [B] is still an open conjecture. For the Fighter with continuous ammunition, \citet{Bartroff10c} show that [A] and [C] hold for the \textit{Frail Fighter} where $u=0$ (implying that if, at an encounter, the fighter fails to shoot down the enemy, he himself is shot down with probability one), and [B] and [C] are shown to hold for the \textit{Invincible Fighter} where $u=1$, thus he can never be shot down.

The Fighter problem with an infinite time horizon and continuous ammunition~$x$ is considered in \citet{Shepp91}, although the terminology there is different.  In this case there is no dependence on $t$ and the optimal allocation can be denoted $\kappa(x)$.  Shepp et al.\ show that for the frail case, where $u = 0$, [B] fails.

In the present paper we consider the discrete ammunition case for the general fighter, $0 \le u \le 1$.  When not unique, we
single out the optimal spending policy $K(n,t)$ which, at each stage, uses the
smallest possible number of missiles; see (\ref{6}). For this $K(n,t)$, we show in Section~\ref{sec:C} that [C] holds for all $0 \le u \le 1$.  In Section~\ref{sec:ABIF} we prove both [A] and [B] for the Invincible Fighter, and we show that for this case $K(n,t)$ takes on all values $n, n-1, \ldots, 1$ as $t$ ranges from zero to infinity.  In Section~\ref{sec:AFF} we prove [A] for the Frail Fighter and show that [B] fails already for $n=5$.  Section~\ref{sec:model} introduces the necessary notation.

We apologize for using militaristic terminology but we did not want to change the terminology already existing in the literature.  We are convinced that the models are general enough to also have non-militaristic applications.

\section{The Model}\label{sec:model}

Let $0=a(0)<a(1) < a(2) < \ldots$ be a concave sequence of probabilities, where $a(j)$ is the  probability that the enemy will be shot down when $j$ missiles are shot at him.  An example is $a(j) = 1-q^j$, where $0< q<1$.  Define the sequences $N_n (j, t)$ and $N(n, t)$ inductively as follows.  $N_n(j, t)$ is the expected number of enemy airplanes shot down when \textit{confronting an enemy presently} and spending $j$ out of the $n$ missiles available on it, and later continuing optimally, and $N(n, t)$ is the total optimal expected number of enemy airplanes shot down when in state $(n,t)$, and confronting an enemy presently. Then
\beq\label{1}
N_1(1, t) = a(1) = N(1, t), \quad N(0, t) = 0,
\eeq
and define $N_n  (j, 0) = a(j)$ for $j = 1, \ldots, n$.
Let the present survival probability, when using $j$ missiles, be denoted $c(j)$.  Clearly
\beq\label{2} c(j) = a(j) + u(1-a(j))= a(j)(1-u)+ u.
\eeq
Then
\beq\label{3}
N_n (j, t) = a(j) + c(j) N^* (n-j, t)
\eeq
where, for $r=1, 2,\dots $,
\beq\label{4}
N^* (r, t) = \int^t_0 N(r, t-v) e^{-v} d v = e^{-t}\int^t_0 N(r, x) e^x dx
= EN (r, t-X),\eeq
where $X$ has an exponential distribution and where we have set $N(r, t) = 0$ for $t < 0$.
Then inductively we have
\beq\label{5}  N(n, t) = \max\limits_{j=1, \ldots, n} N_n (j, t).\eeq
An optimal  spending strategy is therefore
\beq\label{6}
K(n, t) = \min\{ j: N_n (j, t) = \max\limits_{i = 1, \ldots, n} N_n (i, t)\}\eeq
We have written ``an" and not ``the"  since the maximum in (\ref{6}) may possibly be attained for more than one value of $j$, in which case we always choose the smallest, avoiding any issues of uniqueness and making $K(n,t)$ well defined. In what follows, $K(n,t)$ always refers to (\ref{6}).
 
It follows from \eqref{3} that for $j=1,\ldots, n$,
\beq\label{7}
N_{n+1} (j+1, t) = \frac{c(j+1)}{c(j)} [N_n (j, t) - a(j)] + a(j+1),\eeq
which can also be written as
\beq\label{8}
N_n (j, t) = a(j) +\frac{c(j)}{c(j+1)} [N_{n+1} (j+1, t) - a(j+1)].\eeq

\section{Proof of [C]}\label{sec:C}

\begin{lem}\label{lem:1} For $j=1, 2, \ldots $ let
\[ v(j) = a(j) - \frac{c(j)}{c(j+1)} a(j+1).\]
Then
\beq\label{9} v(j) \le v(j+1) \le 0,\eeq
and equality holds (in either or both places) if and
only if $u=0$.
\end{lem}
\begin{proof}
When $u = 0$ the equality in \eqref{9} is immediate, as $c(j) = a(j)$ in this case.  Thus assume $u > 0$.
 The right hand inequality in (\ref{9}) is immediate.
 Simplifying, it is easily seen that the other inequality in \eqref{9} can be written as
\beq\label{10}
\frac{a(j+1) c(j) - a(j) c(j+1)}{a(j+2) c(j+1) - a(j+1) c(j+2)} \ge \frac{c(j+1)}{c(j+2)}.
\eeq
Using \eqref{2} we see that the left hand side of (\ref{10}) simplifies to $[a(j+1) - a(j)]/ [a(j+2) -a(j+1)]$,
and thus does not depend on $u$.  Furthermore, by concavity, the value of the left hand side is always greater than 1, while the right hand side of \eqref{10} depends on $u$ and is always less than, or equal to, one.  Thus \eqref{10} holds.
\end{proof}

We can rewrite \eqref{8} as
\beq\label{11}  N_n (j, t) = \frac{c(j)}{c(j+1)} N_{n+1} (j+1, t) +v(j)\eeq where $v(j) \le 0$.
\begin{thm}\label{thm:C} [C] holds for the general Fighter problem, i.e., for all $0\le u \le 1$.
\end{thm}
\begin{proof}
We must show that $n+1-K(n+1, t) \ge n-K(n, t)$, which can also be written as
\beq\label{12} K (n+1, t) \le K (n, t) +1.\eeq
Clearly \eqref{12} holds when $K (n+1, t) = 1$.  Thus assume  $K (n+1, t) = k$ for some $2 \le k \le n +1$ and some $t$.
By \eqref{11} it follows that
\begin{align} \label{13}
N_n (j, t) &= \frac{c(j)}{c(j+1)} N_{n+1} (j+1, t) + v(j)\nonumber\\
&\le \frac{c(j)}{c(j+1)} N_{n+1} (k, t) + v(j)\\
&= \frac{c(j)}{c(j+1)}\frac{c(k)}{c(k - 1)}\left[N_n(k-1, t) - v(k-1)\right] +v(j)\nonumber
\end{align}
where the inequality in \eqref{13} follows because $k$ is optimal for state $(n+1, t)$ and the last equality uses \eqref{11}.
Note that $k > 1$ was used to avoid $k - 1=0$. The maximum over $j=1,\ldots, n$ on the left hand side of \eqref{13} is, by definition, attained for $j = K(n,t)$.  Note however that for all $j< k-1$ the value of the right hand side of \eqref{13} is, by Lemma~\ref{lem:1} and concavity of the $a(j)$ sequence, less than $N_n(k-1,t)$.  For $j = k-1$ equality holds throughout in \eqref{13}.  Thus the maximum over $j=1, \ldots, n$ of the left hand side cannot be attained for any $j < k-1$, i.e., $K(n,t) \ge k-1$, which is \eqref{12}.
\end{proof}

A similar approach can be used to prove [C] for continuous ammunition for all $0\le u\le 1$. In \citet{Bartroff10c}, [C] was 
only shown to hold for $u = 0$ and $u = 1$.

\section{Proof of [A] and [B] for the Invincible Fighter}\label{sec:ABIF}

We need the following lemma, valid for any $0 \le u \le 1$.
\begin{lem}\label{lem:3} If, for $n\ge 2$,
\beq\label{14}
N_n (j, t) - N_n (j+1, t)\quad\mbox{is non-decreasing in $t$ for $j = 1, \dots, n-1$,} \eeq
 then [A] holds for $n$.
\end{lem}

\begin{proof} Suppose the lemma is false. Then there exist values $t'> t>0$ and $n\ge j > r \ge 1$ such that $K(n, t') = j > K (n, t) = r$.  But then
\[ N(n, t') = N_n (j, t') > N_n (r, t') \; \; \mathrm{and} \; \; N(n, t) = N_n (r, t) \ge N_n (j, t).\]
Thus
\begin{align*}0 &\le N_n (r, t) - N_n (j, t) = \sum^{j-1}_{i = r} [ N_n (i, t) - N_n (i+1, t)] \\
&\le \sum^{j-1}_{i = r} [ N_n (i, t') - N_n (i+1, t') ] = N_n (r, t') - N_n (j, t') < 0,\end{align*}
which clearly is a contradiction.
\end{proof}

Note that for the Invincible Fighter, (\ref{3}) simplifies to
\beq\label{15} N_n (j, t) = a(j) + N^* (n-j, t).\eeq
\begin{lem}\label{lem:4} For the Invincible Fighter and any $n\ge 2$, a sufficient condition for
$ N_n (j, t) - N_n (j+1, t)$ to be strictly increasing in $t$, for $j=1, \dots, n-1$, is that
\beq\label{16} N(s, t) - N(s-1, t)\quad\mbox{is non-decreasing in $t$ for $s=1,\dots, n-1$.}\eeq
\end{lem}
\begin{proof}  By \eqref{15} and \eqref{4},
\begin{multline}\label{17} 
\frac{d}{dt} [N_n (j, t) - N_n (j+1, t)] = \frac{d}{dt} [N^* (n-j, t) - N^* (n-j-1, t)]\\
=[ N(n-j, t) - N(n-j-1, t)] - [N^* (n-j, t) - N^* (n-j-1, t)].
\end{multline}
Now if \eqref{16} holds,
\begin{multline*}
N^* (n-j, t) - N^* (n-j-1,t) = e^{-t}\int\limits^t_0 [N(n-j, v) - N(n-j-1,v)] e^v dv\\
\le e^{-t} [N(n-j, t) - N(n-j-1, t)]\int^t_0 e^v dv= [N(n-j,t) - N(n-j-1, t)] (1-e^{-t})\\
 < N(n-j, t) - N(n-j-1, t)
\end{multline*}
where the strict inequality follows as $N(s, t) - N(s-1, t) > 0 $ for all $s \ge 1$ and $t\ge 0$.  Thus the value in \eqref{17} is strictly positive.\end{proof}

It is important to note that \eqref{16} needs to hold only for $s$ up to $n-1$, as $j \ge 1$.  [A] will then hold for $n$, by Lemma~\ref{lem:3}.  Also note that when \eqref{16} holds, $N^* (s, t) - N^* (s-1, t)$ is strictly increasing in $t$.

\begin{thm}\label{thm:AIF}
[A] holds for the Invincible Fighter.
\end{thm}
\begin{proof} For $n=2$ and $s=1$ in (\ref{16}) we have $N(1, t) - N(0,t ) = a(1)$, which is non-decreasing in $t$, thus [A] holds for $n=2$.
For $s=2$, $N(2, t) - N(1, t) = \max \{ a(2) - a(1), a(1)(1-e^{-t})\}$ and as both of the expressions in the curly brackets are non-decreasing in $t$, it follows that [A] holds for $n=3$. (In our proof by induction we shall need to go back two steps, thus the need to show for $n=2$ and $n=3 $ directly.  A similar fact was overlooked in the proof in \citet{Samuel70}).

Now suppose that \eqref{16} holds for all $n\le m$, where $m \ge 3$.  We shall show that it holds for $n=m+1$, i.e., we need to 
show that $N(m, t) - N(m-1, t)$ is non-decreasing in $t$.  By the induction hypothesis we know that [A] holds for $n = m$ and for $n=m-1$ and that $K(m, t)$ and $K(m-1, t)$ are constant on left closed, right open intervals.  Suppose that in some interval of $t$-values $K(m, t) = j$ and $K(m-1, t)= r$.  Then $1\le j \le r + 1$ by [C], and hence in that interval,
\begin{multline}
\label{18}
N (m,t) - N(m-1,t) = N_m (j, t) - N_{m-1} (r, t)\\
=N_m (j, t) - N_m (r+1, t) + a(r+1) - a(r),
\end{multline}
where the equality follows from \eqref{15}.  For $j=r+1$ the value in \eqref{18} is constant, hence non-decreasing.  For $1\le j \le r $ write
\[
N_m (j, t) - N_m (r+1, t) = \sum^r_{i = j}  [N_m (i, t) - N_m (i+1, t)],\]
and note that all values of the square brackets in the sum are non-decreasing. This is a direct consequence of the induction hypothesis (which assumed that \eqref{16} holds for $n=m$) and Lemma~\ref{lem:4}. It follows that \eqref{16} holds for $n=m+1$ and thus, by Lemma~\ref{lem:3}, that [A] holds for $n=m+1$.
\end{proof}

We now show [B] for the Invincible Fighter.  The proof of the following lemma is similar to that of Lemma~\ref{lem:3}, and hence omitted.

\begin{lem}\label{lem:6} Suppose $N_n (j+1, t) - N_n (j, t)$ is non-decreasing in $n$ for $j = 1,\dots, n-1$ and $t$ fixed.  Then $K (n, t)$ is non-decreasing in $n$, i.e., [B] holds.
\end{lem}
\begin{lem}\label{lem:7} Suppose that $N(s, t)$ is (strictly) concave in $s$ for $s = 0, 1, \ldots, k$ and all fixed $t$.
Then so is $N^* (s, t)$.
\end{lem}
\begin{proof}  Immediate from \eqref{4}.
\end{proof}

The next two lemmas are slight modifications of Lemmas~3.1 and 5.1 of \citet{Bartroff10c}, the differences being  (i) 
the domains of
the functions here are non-negative integers, and (ii) when the definition of
$g(s)$ is for $s$ up to $k$, the resulting definition of $h(s)$ is for $s$
up to $k+1$. The proofs are omitted.


\begin{lem}\label{lem:8}
Let $f(j)$, $j=1,2,\dots$, be a non-negative log-concave sequence and let $g(s)$, $s=0,1,\dots, k$, be a non-negative log-concave function.  For any given $n$, $n = 1, 2, \dots, k+1$, define
\[
h(n) = \max\limits_{j=1,\ldots, n} \{ f(j) g(n-j)\}.\]
Then $h(n)$ is log-concave in $n$, for $n=1,\dots, k+1$.
\end{lem}

\begin{lem}\label{lem:9} Let $f^*(j)$, $j=1,2,\dots$, be a non-negative concave sequence and $g^*(s)$, $s = 0, 1, \dots, k$, be non-negative concave in $s$. Define
\[ h^*(n) = \max\limits_{j=1, \dots, n} \{ f^* (j) + g^* (n-j)\}.\]
Then $h^*(n)$ is concave in $n$ for $n=1,\dots, k+1$.
\end{lem}

\begin{thm}\label{thm:BIF} [B] holds for the Invincible Fighter.
\end{thm}
\begin{proof} We shall use Lemma~\ref{lem:9} and induction on $k$.  Let $f^*(j) = a(j)$ for $j=1, 2, \dots $, which is a non-negative concave sequence, and let $g^*(n-j) = N^*(n-j,t)$, where $t$ is fixed.  Then the resulting $h^*(n)$ will, by \eqref{15}, be $N(n, t)$.  Since $N^*(0, t) = 0$ and $N^*(1, t) = a(1) (1-e^{-t})$, it follows that $g^* (s)$ is trivially non-negative concave for $s=0, 1$, i.e., for $k=1$.  The resulting $N(n, t)$ is therefore concave in $n$ for $n=1, 2$.  Now use Lemma~\ref{lem:7} to imply that $N^*(s, t)$ is concave for $s$ up to $2$, which again by Lemma~\ref{lem:9} implies that $N(s, t)$ is concave for $s$ up to $3$, etc.  The sequence $N(n, t)$ is therefore concave for all $n$ and fixed $t$, which, by Lemma~\ref{lem:7}, implies that the $N^*(n, t)$ sequence is concave.  We shall see that the condition of Lemma~\ref{lem:6}, and hence [B], hold.  Fix $j$ and let $n > j$.  Then
\[
N_n (j+1,t) - N_n (j,t) = a(j+1) - a(j) + N^* (n-j-1, t) - N^*(n-j, t),\]
which is non-decreasing in $n$ by the concavity of the $N^*(n, t)$ sequence.
\end{proof}
Next we show that for the Invincible Fighter, $K(n,t)$ takes on all values $n, n-1,\dots, 1$ as $t$ ranges from zero to infinity.  No similar result has been shown for the Bomber problem.

For $s=1, 2,\dots$ let
\beq\label{19}
D^*(s,t) = N^*(s, t)-N^*(s-1, t).
\eeq
\begin{lem}\label{lem:11}  For the Invincible Fighter, for fixed $n$ and $j=1,\dots, n-1$, the equation
\beq\label{20} N_n (j+1, t) = N_n (j, t)\eeq
has a unique solution, to be denoted $t(n, j)$.  For $t< t(n,j)$ the inequality $N_n(j+1,t) > N_n(j, t)$ holds, whereas for $t > t(n,j)$ one has $N_n(j+1,t) < N_n (j,t)$.
\end{lem}

\begin{proof} Equation \eqref{20} can be written as \beq\label{21} D^*(n-j,t) = a(j+1) - a(j).\eeq In the proof of [A] we
showed that $D^*(s,t)$ is strictly increasing in $t$.  Thus if we establish the existence of $t(n,j)$, the rest of the claim
follows.  We shall show that \begin{description} \item[(a)] For any $s$, $\lim\limits_{t\to 0} D^*(s, t) = 0$, and \item[(b)]
For any $s$, $\lim\limits_{t\to\infty} D^*(s, t) = a(1)$. \end{description} Claim (a) is trivially true, as the domain of the
integration defining $N^*(s, t)$ tends to $0$.  For (b), fix $s$ and let $p(t)$ denote the probability of at least $s$ enemy
planes in the time period $(0,t)$, i.e., $p(t)=e^{-t}\sum_{i\ge s}t^i/i!\To 1$ as $t\To\infty$. Then $N^*(s,t)$ is bounded
below by the expected number obtained by spending one unit of ammunition per each of the first $s$ enemies; neglecting the
number shot down given fewer than $s$ arrivals, this gives $N^*(s,t)\ge sa(1)p(t)\To sa(1)$ as $t\To\infty$. On the other hand,
if the Fighter somehow knew there would be at least $s$ arrivals, it follows from concavity of $a(j)$ that the optimal strategy 
in
state $(s,t)$ would be to spend one unit on each of the first $s$ enemies. Crudely bounding the number shot down given fewer than $s$
arrivals by $s$, we have $N^*(s,t)\le s(1-p(t))+sa(1)p(t)\To sa(1)$ as $t\To\infty$, showing that $N^*(s,t)\To sa(1)$, and (b)
now follows from \eqref{19}.  Now (a) and (b) guarantee the existence of a solution to \eqref{21} by continuity of $D^*(s, t)$
in $t$, and by concavity which implies $0< a(j+1) - a(j) < a(1)$. \end{proof} \begin{thm}\label{thm:K} For the Invincible
Fighter, $K(n,t)$ takes all the values $n, n-1, \dots, 1$ as $t$ ranges from $0$ to $\infty$. \end{thm} \begin{proof} In 
the
notation of Lemma~\ref{lem:11} it suffices to show that $t(n,j+1) < t(n,j)$ for $j=0,\dots, n-1$, as clearly then $K(n,t) =
j+1$ for $t(n,j+1) \le t < t(n,j)$, where we have let $t(n,n) =0$ and $t(n,0)=\infty$.  Actually we shall show the stronger
conclusion \beq\label{22} t(n,j+1) < t(n+1,j+1) < t(n,j)\; \; \mathrm{for\ } \; \;  j = 0,\dots, n-1,\eeq which shows that
$K(n+1, t)$ will change exactly once in every interval of constancy of $K(n,t)$.

The first inequality in \eqref{22} follows using \eqref{21} and the fact that $D^*(s,t)$ is strictly decreasing in $s$ and $t$, the former by the strict concavity of the $N^*(s, t)$ sequence for every fixed $t$, shown in the proof of Theorem~\ref{thm:BIF}.  The second inequality follows using \eqref{21} and the fact that the sequence $\{a(j)\}$ is strictly  concave.
\end{proof}

\section{Proof of [A] and Counterexample to [B] for the Frail Fighter} \label{sec:AFF}

The ideas of the proofs in the present section are similar to those of Section~\ref{sec:ABIF}.  We shall need the following lemma, which is a slightly changed version of the lemma in \citet{Samuel70}.
\begin{lem}\label{lem:13} Fix $n$ and assume that $f(i, t)$ are positive continuous functions of $t$ for $i=1,\dots, n$ such that $f(i, t)/f(i+1, t)$ are monotone non-decreasing in $t$ for $i=1\dots, n-1$.  Let
\[
K(n, t) = \min\{ j: f(j,t) = \max\limits_{i=1, \dots, n} f(i, t)\}.\]
Then $K(n, t)$ is non-increasing in $t$ and right continuous, i.e., [A] holds for $n$.
\end{lem}
We omit the proof, as it is verbatim the same as that in \citet{Samuel70}.  (The statement there required that $f(i, t)/f(i+1, t)$ be strictly increasing, but this was not used in the proof.)
\begin{thm} \  [A] holds for the Frail Fighter.
\end{thm}
\begin{proof} We would like to show, for $f(i, t) = N_n(i, t)$ of Lemma~\ref{lem:13}, that $N_n (i, t)/N_n(i+1, t)$ is non-decreasing in $t$ for $i=1, \dots, n-1$.  This is immediate for $i=n-1$ as clearly $N_n (n-1, t) / a(n)$ is non-decreasing in $t$.  Thus [A] holds for $n=2$ (and all $u$).  For the Frail Fighter, (\ref{3}) simplifies to
\beq \label{23} N_n (i, t) = a(i) [1+N^* (n-i, t)], \quad i = 1,\dots, n.\eeq
Thus for the conditions of Lemma~\ref{lem:13} to hold we must show that
\beq\label{24}
[1+N^* (n-i, t)] / [1+N^* (n-i-1, t)] \; \; \mathrm{for \ } \; \; i = 1, \dots, n-2\eeq
is non-decreasing in $t$.
Since $i\ge 1$ and \eqref{24} is trivial for $i=n-1$, \eqref{24} would imply that [A] holds for $n$, where $n\ge 3$.  We shall show that $N^* (k,t)$ is \textit{totally positive of order 2} (TP$_2$, see \citet{Karlin68}), i.e., that the ratio $N^*(k, t)/N^* (k-1, t)$ is non-decreasing in $t$ for $k=2, \dots, n-1$. Since $N^*(k, t)$ is non-decreasing in both variables, it follows similarly to the proof of Lemma 3.5 in \citet{Bartroff10c} that also $1+N^*(k, t)$ is TP$_2$, i.e. that \eqref{24} holds.  After some cancelation it follows that
\beq\label{25} N^*(k-1, t)^2 \frac{d}{dt} \left[\frac{N^* (k, t)}{N^*(k-1, t)}\right] = N(k, t) N^* (k-1, t) - N(k-1, t) N^*(k, t)
\eeq
and we would like to show that the value in the right hand side is non-negative.
Suppose we show that $N(k, t)/N(k-1, t)$ is non-decreasing in $t$.  Then we would have that
\beq\label{26} N(k, t)N(k-1, t-X) \ge N(k, t-X) N(k-1, t)\eeq
where $X$ is exponentially distributed, and where we define $N(k, t) = 0$ for $t < 0$.  Taking expectations over $X$ in \eqref{26} yields that the right hand side in \eqref{25} is non-negative.  Thus it suffices to show that
\beq\label{27} \frac{N(k, t)}{N(k-1, t)}\quad\mbox{is non-decreasing\ in $t$ for $k=2, \ldots, n-1$}\eeq
for [A] to hold for $n$.  For $k=2$ the claim is trivial, thus [A] holds for $n=3$.  Now suppose \eqref{27} holds for $n=m$ 
where $m\ge 3$.  We shall show that it holds for $n=m+1$. Note that by \eqref{27} the only additional requirement is to 
show 
that $N(m,t)/N(m-1, t)$ is non-decreasing in $t$.   By the induction hypothesis we know that [A] holds for $m$ and for $m-1$, and that $K(m,t)$ and $K(m-1, t)$ are constant on left closed, right open intervals.  Suppose that in some interval, $K(m, t) = j$ and $K(m-1, t) = r$.  Then, by [C], $1\le j \le r+1 \le m$ and
\beq\label{28} \frac{N(m, t)}{N(m-1, t)} = \frac{N_m (j, t)}{N_{m-1} (r, t)} = \frac{N_m (j, t)}{N_m (r+1, t)} \cdot \frac{a(r+1)}{a(r)}\eeq
which is constant in $t$, hence non-decreasing, when $j= r+1$.

For $j \le r$ write
\beq\label{29} \frac{N_m (j,t)}{N_m (r+1, t)} = \prod\limits^r_{i=j} \frac{N_m (i, t)}{N_m (i+1, t)}.\eeq
Now \eqref{27} holds for $n=m$ by the induction hypothesis and it follows that each term in the product \eqref{29} is non-decreasing in $t$, and hence the left hand side of \eqref{28} is non-decreasing in $t$, implying that  [A] holds for $n=m+1$.
\end{proof}

 \noindent {\bf Remark.} Unlike Theorem~\ref{thm:K}, a similar statement fails to hold for the Frail Fighter.  Even in the simple case $a(j) = 1-q^{j}$  and $n=2$ it is easily seen that spending both missiles is preferable to spending one for any $t$ whenever $1/2 < q \le 1$.  More generally, spending $n$ is preferable to spending $n-1$ for any $t$ whenever $q> (1/2)^{1/{(n-1)}}$.

\begin{asser}\label{counterexample} \conjB\; fails for the Frail Fighter.  
\end{asser}
\begin{proof} Since we have shown that [A] holds, the optimal $K(n,t)$ can be found by comparing $N_n(j+1,t)$ and $N_n(j,t)$ 
for 
$j=n-1,\ldots,1$. For the Frail Fighter with $a(j)=1-q^j$ and $q=1/2$, and letting $I(\cdot)$ denote an indicator function, one finds that for $0\le t<\infty$,
\begin{align*}
K(2,t)&=2, \\
K(3,t)&=3I(t<\log(3/2))+2I(t\ge\log(3/2)),\\
K(4,t)&=4I(t<\log(7/6))+3I(t\ge\log(7/6)), 
\end{align*}
and finally,
\begin{multline*}
K(5,t)=5I(t<\log(15/14))+4I(\log(15/14)\le t<\log(3/2))\\
+3I(\log(3/2)\le t<2.694\ldots)+2I(t\ge 2.694\ldots), 
\end{multline*}
where $2.694\ldots$ is the numerical solution to 
\begin{equation}
\frac{3}{4}\left[\frac{17}{8}-e^{-t}\left(\frac{5}{4}+\frac{3}{8}\left(t-\log(3/2)\right)\right)\right]=\frac{7}{8}\left[1+\frac{3}{4}(1-e^{-t})\right].
\end{equation}
Thus for $t> 2.694\ldots$ we have $K(4,t)=3$ while $K(5,t)=2$, violating [B]. Note that, by continuity in $u$, it is clear that a similar result (with slightly different cut-off point) will hold for $u>0$ sufficiently small. 
\end{proof}

\section*{Acknowledgements}
We would like to thank Larry Goldstein and Yosef Rinott for many helpful discussions on the Fighter and Bomber problems. Bartroff's work was supported in part by grant DMS-0907241 from the National Science Foundation and a Faculty-in-Residence grant from the Albert and Elaine Borchard Foundation.


\def\cprime{$'$}

\end{document}